\documentclass[12pt]{amsart}
\usepackage{amsmath,amsthm,amsfonts,amssymb,latexsym,verbatim,bbm}
\input xy
\xyoption{all}
\usepackage{hyperref,multirow}
\usepackage[shortlabels]{enumitem}

\allowdisplaybreaks[2] \textwidth15.1cm \textheight22cm \headheight12pt \oddsidemargin.4cm
\def\mhline{\noalign{\ifnum0=`}\fi\hrule height 4\arrayrulewidth \futurelet
   \@tempa\oxhline}
\def\oxhline{\ifx\@tempa\hline\vskip \doublerulesep\fi
      \ifnum0=`{\fi}}

\makeatother

\numberwithin{equation}{section}

\theoremstyle{plain}
\newtheorem{theorem}{Theorem}[section]
\newtheorem{thm}[theorem]{Theorem}
\newtheorem{lemma}[theorem]{Lemma}

\newtheorem{prop}[theorem]{Proposition}
\newtheorem{example}[theorem]{Example}

\newtheorem{defn}[theorem]{Definition}

\newtheorem{cor}[theorem]{Corollary}

\newtheorem{rmk}[theorem]{Remark}

\newcommand{\iso}{\cong}
\newcommand{\arr}{\rightarrow}

\DeclareMathOperator{\rank}{rank}

\newcommand{\R}{\mathbb{R}}
\newcommand{\C}{\mathbb{C}}
\newcommand{\Z}{\mathbb{Z}}

\newcommand{\mcA}{\mathcal{A}}

\newcommand{\mcC}{\mathcal{C}}

\newcommand{\mcG}{\mathcal{G}}

\newcommand{\mcM}{\mathcal{M}}

\newcommand{\mfg}{\mathfrak{g}}

\newcommand{\mfn}{\mathfrak{n}}

\DeclareMathOperator{\vspan}{span}

\DeclareMathOperator{\mfb}{\mathfrak{b}}

\DeclareMathOperator{\Der}{Der}

\DeclareMathOperator{\Exp}{Exp}

\begin{document}

\title{A pattern avoidance criterion for free inversion arrangements}

\author{William Slofstra}
\email{wslofstra@math.ucdavis.edu}

\begin{abstract}
    We show that the hyperplane arrangement of a coconvex set in a finite
    root system is free if and only if it is free in corank $4$. As a
    consequence, we show that the inversion arrangement of a Weyl group element
    $w$ is free if and only if $w$ avoids a finite list of root system
    patterns. As a key part of the proof, we use a recent theorem of 
    Abe and Yoshinaga to show that if the root system does not contain
    any factors of type $C$ or $F$, then Peterson translation of coconvex
    sets preserves freeness. This also allows us to give a
    Kostant-Shapiro-Steinberg rule for the coexponents of a free inversion
    arrangement in any type.
\end{abstract}
\maketitle

\section{Introduction}

A central hyperplane arrangement $\mcA$ in a complex vector space
$\C^l$ is said to be free if its module of derivations $\Der(\mcA)$ is free. If
$\mcA$ is free, then $\Der(\mcA)$ has a homogeneous basis of size $l$, and the
degrees $d_1,\ldots,d_l$ of the generators are called the coexponents of
$\mcA$. By a Theorem of Terao, the Poincare series $\sum t^i \dim H_i (\C^l
\setminus \mcA)$ of the complement of $\mcA$ is then $\prod_i (1+d_i t)$
\cite{Te81}.  The best known free arrangements are the Coxeter arrangements,
cut out by the root systems of finite Coxeter groups, for which the coexponents
are equal to the exponents of the Coxeter group \cite{Te81}.

It is natural to ask when a subarrangement of a Coxeter arrangement is free.
For braid arrangements (i.e.~ Coxeter arrangements of type $A$), the
subarrangements are graphic arrangements, and a theorem of Stanley states that
a graphic arrangement is free if and only if the corresponding graph is chordal
\cite{ER94} \cite{CDFSSTY}. For other types, the only known results focus on certain
classes of subarrangements \cite{ER94} \cite{ST06} \cite{ABCHT13} \cite{OPY08} \cite{Sl13}. In particular if $R$ is a finite crystallographic
root system, and $S \subseteq R^+$ is a lower order ideal in dominance order,
then the arrangement $\mcA(S)$ cut out by $S$ is free by a theorem of Sommers
and Tymoczko (all types except type $E$ \cite{ST06}) and Abe, Barakat, Cuntz,
Hoge, and Terao (all types \cite{ABCHT13}). In this case, the coexponents of
$\mcA(S)$ can be determined from the heights of the roots in $S$. If $w$ is a
rationally smooth element of the Weyl group $W(R)$, then the arrangement
$\mcA(I(w))$ cut out by the inversion set $I(w)$ of $w$ is also free, by a
theorem of Oh, Postnikov, and Yoo (type A \cite{OPY08}) and the author (all
types \cite{Sl13}). In this case, the coexponents of $\mcA(I(w))$ are equal to
the exponents of $w$. 

Let $X$ be the flag variety of type $R$, and let $X(w)$ be the Schubert variety
indexed by $w \in W(R)$. The maximal torus $T$ of the underlying semisimple
group acts on $X(w)$, and the $T$-fixed points are the elements $x$ of $W(R)$
which are less than or equal to $w$ in Bruhat order. The $T$-weights of the tangent
space $T_w X(w)$ are the inversions $I(w) \subseteq R^+$ of $w$, while the
$T$-weights in the tangent space $T_e X(w)$ at the identity $e \in W(R)$
correspond to a lower order ideal $S(w) \subseteq R^+$. If $X(w)$ is smooth, then
$I(w)$ and $S(w)$ have the same number of elements, and both $\mcA(I(w))$ and
$\mcA(S(w))$ are free. A theorem of Akyildiz and Carrell \cite{AC12} states
that the coexpoents of $\mcA(S(w))$ are equal to the exponents of $w$, so in
the smooth case the coexponents of $\mcA(S(w))$ are equal to the coexponents of
$\mcA(I(w))$. This leads to the question of whether $\mcA(I(w))$ and
$\mcA(S(w))$ can be compared directly.

In this paper, we give an affirmative answer to this question using Peterson
translation. If $x \leq y \leq w$, where $x = r_{\alpha} y$, then there is a
$T$-invariant curve in $X(w)$ connecting $x$ and $y$, and any $T$-submodule $M
\subseteq T_y X(w)$ can be translated along this curve to a $T$-submodule of $T_x
X(w)$ \cite{CK03}. This geometric Peterson translation procedure defines a
combinatorial Peterson translate on coconvex sets of $R^+$. We say that a
coconvex set $S$ is Peterson-free if $\mcA(S)$ is free, and $\mcA(\tau(S))$ is
free for all Peterson translates $\tau(S)$ of $S$. All lower order ideals are
trivially Peterson-free, and an inversion set is Peterson-free if and only if
it is free (the same is true for coconvex sets as long as $R$ contains no
factors of type $C$ or $F$). What's more, the class of Peterson-free coconvex
sets is closed under Peterson translation, and thus any Peterson-free coconvex
set can be translated to a lower order ideal with the same coexponents.
Consequently if $X(w)$ is smooth then we can directly compare $\mcA(I(w))$ and
$\mcA(S(w))$ by repeatedly translating $I(w)$ from $w$ to the identity.  We
extend these techniques further to show that the arrangement $\mcA(S)$ cut out
by a coconvex set $S$ is free if and only if the localizations $\mcA(S)_X$ are
free for every flat of $\mcA(S)$ of corank $\leq 4$. This gives a root-system
pattern avoidance criterion for the freeness of $\mcA(I(w))$. The proofs of
these results are for the most part short, relying on a theorem of Abe
and Yoshinaga \cite{AY13} to reduce to low rank, where results can be
checked by computer. The exception is type $C_n$, where we rely on a
characterization of free subarrangements containing $A_{n-1}$ due to Edelman
and Reiner \cite{ER94} to circumvent the fact that not all free coconvex
sets in $C_n$ are Peterson-free. 

The rest of the paper is organized as follows. In Section \ref{S:pattern}, we
state pattern avoidance criterions for the freeness of arrangements cut out by
coconvex and inversion sets. In Section \ref{S:peterson} we state and
prove the main properties of combinatorial Peterson translation. In Section
\ref{S:ptgraph} we develop some additional tools, such as Peterson-freeness, to
study Peterson translation in types $C$ and $F$. In Section \ref{S:proofs} we
prove the pattern avoidance results, and explain some of the technical aspects
of the computer checks used throughout the paper. Finally, in Section
\ref{S:geometry} we explain how the combinatorial results are related to the
geometric Peterson translate. 

\subsection{Acknowledgements}
I thank Jim Carrell for suggesting that Peterson translation could be applied
to inversion arrangements, Ed Richmond for many interesting conversations on
Peterson translation, and Alejandro Morales, Stefan Tohaneanu, and Alex Woo for helpful
comments.

\subsection{Notation}
$R$ will refer to a finite crystallographic root system in ambient Euclidean
space $V$, with positive and negative roots $R^+$ and $R^-$ respectively.  We
assume without loss of generality that $R$ spans $V$. Given $S \subseteq R$ with
the property that $S \cap -S = \emptyset$, we let $\mcA(S)$ denote the real
arrangement in $V^*$ cut out by $S$. 

$L(\mcA)$ will denote the intersection lattice of an arrangement $\mcA$. Given
$X \in L(\mcA)$, the localization $\mcA_X$ is the subarrangement of $\mcA$
consisting of those hyperplanes which contain $X$.

A real arrangement $\mcA$ in a real vector space $V^*$ has a complexification
$\mcA_{\C}$ in $V^*_{\C} = V^* \otimes \C$. We say that a real arrangement
$\mcA$ is free if and only if $\mcA_{\C}$ is free.  The intersection lattices
$L(\mcA)$ and $L(\mcA_{\C})$ are isomorphic, and complexification commutes with
localization.

\section{Coconvex sets, freeness, and pattern avoidance}\label{S:pattern}

A subset $S$ of $R^+$ is convex if $\alpha, \beta \in S$, $\alpha + \beta \in
R^+$ implies that $\alpha + \beta \in S$, coconvex if $R^+ \setminus S$ is
convex, and biconvex if $S$ is convex and coconvex. The main result of this
paper is a criterion for $\mcA(S)$ to be free when $S$ is a coconvex set.  Let
$X$ be a flat of an arrangement $\mcA$. Since $X$ is the center of $\mcA_X$,
the rank of $\mcA_X$ is equal to the corank of $X$. If $\mcA$ is free, then so
is $\mcA_X$ \cite[Theorem 4.37]{OT92}. For arrangements $\mcA(S)$ of coconvex
sets $S$, we show that this result can be reversed while only checking flats
$X$ of low corank (and hence arrangements $\mcA(S)_X$ of low rank). 
\begin{thm}\label{T:coconvex}
    Let $S$ be a coconvex subset of $R^+$. Then $\mcA(S)$ is free if and only
    if $\mcA(S)_X$ is free for every flat $X \in L(\mcA(S))$ of corank $\leq
    4$. 
\end{thm}
Every arrangement of rank $\leq 2$ is free. If $\mcA(S)$ has rank $\geq 4$,
then $\mcA(S)$ will be free in Theorem \ref{T:coconvex} if and only if
$\mcA(S)_X$ is free for every flat $X$ of corank $4$.  We will show in the proof
that if $R$ contains no factors of types $D$, $E$, or $F$, then $\mcA(S)$ is
free if and only if $\mcA_X$ is free for every flat $X$ of corank $3$.
Theorem \ref{T:coconvex} will be proved in Section \ref{S:proofs}.

If $U$ is a subspace of $V$, then $R_U = R \cap U$ is a root system with
positive and negative roots $R_U^+ = R^+ \cap U$ and $R_U^- = R^- \cap U$
respectively. Let $R_0$ be another root system in an ambient space $V_0$, and
fix $S_0 \subseteq R_0^+$. We say that a subset $S \subseteq R^+$ contains the
pattern $(S_0, R_0)$ if there is a subspace $U$ of $V$ such that $R_U \iso
R_0$, and this isomorphism identifies $S_U = S \cap U \subseteq R^+_U$ with
$S_0$. If $S$ does not contain $(S_0, R_0)$, then we say that $S$ avoids $(S_0,
R_0)$. If a coconvex (resp. convex, biconvex) set $S$ contains a pattern $(S_0,
R_0)$, then $S_0$ is also coconvex (resp. convex, biconvex). We say that a
pattern $(S_0, R_0)$ is coconvex (resp. convex, biconvex) if $S_0$ is coconvex
(resp. convex, biconvex).  

Given an arbitrary subset $S \subseteq R^+$ and a subspace $U$, let $X' =
\bigcap_{\alpha \in U} \ker \alpha$, and let $X$ be the smallest flat of
$\mcA(S)$ containing $X'$ (if $U$ is spanned by elements of $S$ then $X = X'$,
whereas in general $X$ is the sum of $X'$ plus the center of $\mcA(S)$). The
arrangement $\mcA(S_U)$ in $U^*$ is linearly isomorphic to $\mcA(S)_X / X'$, a
localization followed by a quotient, and consequently $\mcA(S_U)$ is free if
and only if $\mcA(S)_X$ is free. Thus if $S$ contains the pattern $(S_0, R_0)$
and $\mcA(S)$ is free then $\mcA(S_0)$ is also free. Conversely if $S$ contains
$(S_0, R_0)$ and $\mcA(S_0)$ is not free then $\mcA(S)$ cannot be free. Theorem
\ref{T:coconvex} states that $\mcA(S)$ is free if and only if $S$ avoids every
non-free pattern $(S_0, R_0)$ in root systems $R_0$ of rank $\leq 4$. To
determine exactly which patterns must be avoided, we look at minimal patterns:
\begin{defn}
    We say that a pattern $(S_0,R_0)$ is a \emph{minimal non-free pattern} if
    $\mcA(S_0)$ is non-free, and there is no proper subspace $U_0$ of the
    ambient space $V_0$ such that $\mcA\left((S_0)_{U_0}\right)$ is a non-free
    arrangement in $U_0^*$. 
\end{defn}
Note that if $(S_0, R_0)$ is minimal, then the vectors in $S_0$ span the ambient
space $V_0$.  Theorem \ref{T:coconvex} can now be phrased in terms of pattern
avoidance.
\begin{cor}\label{C:coconvexpat}
    Let $R$ be a finite crystallographic root system, and let $S \subseteq R^+$
    be a coconvex set. Then $\mcA(S)$ is free if and only if $S$ avoids the
    minimal non-free coconvex patterns $(S_0, R_0)$, all of which are contained
    in the root systems $R_0 = A_3$, $B_3$, $C_3$, $D_4$, and $F_4$.  
\end{cor}

The notion of root system pattern avoidance defined above is inspired by root
system pattern avoidance for Weyl groups.  Let $W(R)$ denote the Weyl group of
$R$. The inversion set $I(w)$ of $w \in W(R)$ is the set $\{\alpha \in R^+ :
w^{-1} \alpha \in R^-\}$.  The inversion set $I(w)$ uniquely identifies $w$,
and a subset $S \subseteq R^+$ is biconvex if and only if $S = I(w)$ for some
element $w \in W(R)$. An element $w \in W(R)$ is said to contain (resp. avoid)
the pattern $(w_0, R_0)$ if $I(w)$ contains (resp.  avoids) the pattern
$(I(w_0), R_0)$. Thus root system pattern avoidance for Weyl groups is
equivalent to the definition of pattern avoidance for biconvex sets given
above. Root system pattern avoidance has been used by Billey and Postnikov
\cite{BP05} to characterize rationally smooth elements in $W(R)$. In type $A$,
root system pattern avoidance is roughly equivalent to the usual pattern avoidance for
permutations. 

We say that $(w_0, R_0)$ is a minimal non-free pattern if $(I(w_0), R_0)$ is
minimal non-free. Since Corollary \ref{C:coconvexpat} holds in particular for
biconvex sets, we have:
\begin{cor}\label{C:biconvexpat}
    Let $R$ be a finite crystallographic root system, and let $w \in W(R)$.
    Then $\mcA(I(w))$ is free if and only if $w$ avoids the minimal non-free
    patterns $(w_0, R_0)$, all of which are contained in the root systems
    $R_0 = A_3$, $B_3$, $C_3$, $D_4$, and $F_4$.
\end{cor}

The number of minimal patterns in Corollaries \ref{C:coconvexpat} and
\ref{C:biconvexpat} are given in Table \ref{TBL:badpat}. The minimal non-free
patterns $(w_0, R_0)$ are explicitly listed in Table \ref{TBL:badpatlist},
where we use the same Dynkin diagram labelling as in \cite{BP05} and
\cite{Sl13}, with $s_i$ referring to the $i$th simple reflection. To save
space in Table \ref{TBL:badpatlist}, we use the notation $[a,b,\ldots]$ to
refer to a list of optional terms, so for example $[a,b]c$ would refer to the
three terms $c$, $ac$, and $bc$. Tables \ref{TBL:badpat} and
\ref{TBL:badpatlist} were constructed using an exhaustive computer search,
which we describe further in Section \ref{S:proofs}.

Given $\alpha \in R$, let $\check{\alpha}$ denote the coroot $2\alpha /
(\alpha,\alpha)$ in the dual root system $\check{R}$. Given $S \subseteq R^+$,
let $\check{S} = \{\check{\alpha} : \alpha \in S\} \subseteq \check{R}^+$. Then
$S$ is a biconvex set in $R$ if and only if $\check{S}$ is a biconvex set in
$\check{R}$, and $\mcA(S) = \mcA(\check{S})$. Similarly if $\sigma$ is a
diagram automorphism of $R$, then $\sigma$ sends biconvex sets to biconvex
sets, and $\mcA(\sigma(S)) \iso \mcA(S)$. These two features can be seen in
Table \ref{TBL:badpatlist}. For example, $B_3$ and $C_3$ are dual, so $W(B_3)
\iso W(C_3)$ have the same minimal non-free biconvex patterns.  The number of minimal
non-free coconvex patterns is different, however. $D_4$ has a diagram
automorphism $\sigma$ of order three, and the elements $s_i s_j s_3 s_2 s_1 s_3
s_4 s_2 s_i s_j$ all lie in the same $\sigma$ orbit. Finally, the root system $F_4$ is
self-dual, with the corresponding automorphism on $W(F_4)$ given by a diagram
automorphism $\sigma'$ of the Coxeter group. The first pattern
listed in Table \ref{TBL:badpatlist} for $F_4$ is invariant under $\sigma'$, while the
last two lie in the same $\sigma'$-orbit. Also note that if $S$ contains a
pattern $(S_0, R_0)$, and $\sigma$ is a diagram automorphism of $S_0$, then $S$
also contains $(\sigma(S_0), R_0)$, so technically there is some redundancy in
listing two minimal non-free patterns in the same $\sigma$-orbit. 

\begin{table}
    \begin{tabular}{ccc}
        $R_0$ & Weyl group elements $w_0$ & Coconvex sets $S_0$ \\
        $A_3$ & 1 & 3 \\
        $B_3$ & 7 & 42 \\
        $C_3$ & 7 & 50 \\
        $D_4$ & 4 & 21 \\
        $F_4$ & 3 & 391 \\
    \end{tabular}
    \caption{The number of minimal non-free patterns $(w_0, R_0)$ (resp.
        $(S_0, R_0)$).} 
    \label{TBL:badpat}.
\end{table}

\begin{table}
    \begin{tabular}{lll}
        $R_0$ & Elements $w_0$\\
        $A_3$ & $s_2 s_1 s_3 s_2$ \\
        $B_3$/$C_3$ & $[s_3] s_2 s_1 s_3 s_2 [s_3]$, $s_2 s_1 s_3 s_2 s_1 s_3 [s_2]$, $s_1 s_3 s_2 s_1 s_3 s_2$ \\
        $D_4$ & $s_2 s_1 s_3 s_4 s_2$, $s_i s_j s_2 s_1 s_3 s_4 s_2 s_i s_j$, $i,j \in \{1,3,4\}$, $i<j$ \\
        $F_4$ & $s_4 s_3 s_2 s_3 s_4 s_1 s_2 s_3 s_4 s_2 s_1 s_3 s_2 s_1 s_3 s_2 s_4 s_3 s_2 s_1$, 
                $s_2 s_1 s_4 s_3 s_4 s_2 s_3 s_1$, $s_3 s_4 s_1 s_2 s_1 s_3 s_2 s_4$ \\
    \end{tabular}
    \caption{List of elements $w_0 \in W(R_0)$ such that $(w_0,R_0)$ is a 
        minimal non-free pattern.}
    \label{TBL:badpatlist}
\end{table}

\begin{example}\label{Ex:typeA}
    Let $e_1,\ldots,e_{n+1}$ denote the standard basis of $\R^{n+1}$. The root
    system $A_n$ is usually presented as $R = \{e_j - e_i : i \neq j\}$ in
    ambient space $V = \vspan\{e_{i+1} - e_{i} : i=1,\ldots,n\}$. In this
    presentation, we take $R^+ = \{e_j - e_i : 1 \leq i < j \leq n+1 \}$.
    Any subset $S \subseteq R^+$ corresponds to a graph $G(S)$ with vertex set
    $\{1,\ldots,n+1\}$ and edge set $\{ij : e_j - e_i \in S\}$. As mentioned
    in the introduction, the arrangement $\mcA(S)$ is free if and only if
    $G(S)$ is chordal, meaning that every cycle of length $\geq 4$ has a
    chord. 

    Let $s_i$ be the simple reflection corresponding to simple root $\alpha_i =
    e_{i+1} - e_{i}$, and let $w = s_2 s_1 s_3 s_2$ in $W(A_3)$. Then
    \begin{equation*}
        I(w) = \{\alpha_2, \alpha_1 + \alpha_2, \alpha_2 + \alpha_3, \alpha_1
            + \alpha_2 + \alpha_3\} = \{e_3 - e_2, e_3 - e_1, e_4-e_2, e_4 - e_1\}.
    \end{equation*}
    The associated graph $G(I(w))$ is 
    
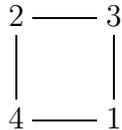
\begin{equation*}
        \xymatrix{ 2 \ar@{-}[r] \ar@{-}[d] & 3 \ar@{-}[d] \\
                         4 \ar@{-}[r] & 1 \\ }
    \end{equation*}
    Thus $G(I(w))$ is not chordal, and $w$ is the unique element of $W(A_3)$
    for which $\mcA(I(w))$ is not free.
\end{example}


\section{Combinatorial Peterson translation}\label{S:peterson} 
In this section we introduce a translation procedure on coconvex subsets $S
\subseteq R^+$ which is analogous to Peterson translation on the tangent space
of a Schubert variety. This procedure is essential to the proof of Theorem
\ref{T:coconvex}, and can be used to calculate coexponents of free inversion
arrangements. However, to see that this procedure is correct in types $C$ and
$F$ we rely on the more sophisticated analysis presented in the next section.

Let $\preceq$ denote the dominance ordering on $R^+$, so $\alpha \preceq \beta$
if and only if $\beta -\alpha$ is a non-negative linear combination of simple
roots. A subset $S \subseteq R^+$ is a lower order ideal of $R^+$ if $\beta \in
S$, $\alpha \preceq \beta$ implies that $\alpha \in S$. All lower order ideals
are coconvex. If $R^+$ has simple roots $\{\alpha_i\}$, and $\alpha = \sum n_i
\alpha_i \in R^+$, then the height of $\alpha$ is $\sum n_i$. Given a lower order
ideal $S$, let $h_i$, $i \geq 1$, denote the number of elements of height $i$,
and set $h_0 = l$, where $l$ is the rank of $R$. In a lower order ideal, the
number $h_i - h_{i+1}$ is always non-negative, so we can define the exponent
set $\Exp(S)$ of $S$ to be a multi-set of non-negative integers in which $i$
appears with multiplicity $h_i - h_{i+1}$. If $S = R^+$ then $\Exp(S)$ is the
set of exponents of $W$ by a theorem of Kostant \cite{Ko59}, and this is
sometimes called the Kostant-Macdonald-Shapiro-Steinberg rule for the
exponents of $W$. As mentioned in the introduction, we will need:
\begin{thm}[\cite{ST06}, \cite{ABCHT13}]\label{T:idealfree}
    If $S$ is a lower order ideal in $R^+$, then $\mcA(S)$ is free with
    coexponents $\Exp(S)$.
\end{thm}
Given $\alpha \in R^+$, an $\alpha$-string is a subset of $R^+$ of the form
$\{\beta, \beta+\alpha, \beta+2\alpha,\ldots,\beta+k\alpha\}$, where
$\beta-\alpha \not\in R^+$ and $\beta+(k+1)\alpha \not\in R^+$. The set of
$\alpha$-strings partitions $R^+$. The Peterson translate of a subset $S
\subseteq R^+$ compresses each $\alpha$-string:
\begin{defn}
    Given $S \subseteq R^+$, $\alpha \in R^+$, we define the Peterson translate
    $\tau(S,\alpha)$ of $S$ by $\alpha$ as follows:
    \begin{itemize}
        \item If $S$ is a subset of an $\alpha$-string $\{\beta, \beta+\alpha,
            \ldots, \beta+k\alpha\}$ in $R^+$, so $S = \{\beta+i_1 \alpha,
            \ldots, \beta+i_r \alpha\}$, then $\tau(S, \alpha) = \{\beta,
            \beta+\alpha,\ldots, \beta+(r-1)\alpha\}$.
        \item For a general subset $S$ of $R^+$, let $S = \bigcup S_i$ be the
            partition of $S$ induced by the partition of $R^+$ into
            $\alpha$-strings. Then $\tau(S,\alpha) = \bigcup \tau(S_i,\alpha)$. 
    \end{itemize}
\end{defn}
In Section \ref{S:geometry} we will show that this definition is equivalent to
a geometric formula for Peterson translation given by Carrell and Kuttler \cite{CK03}.
\begin{example}
    In type $A$, all $\alpha$-strings have size $1$ or $2$. Using the notation
    of Example \ref{Ex:typeA}, the $\alpha_1$-strings can be arranged as 
    \begin{equation*}
        \xymatrix@R=0pt{ & \alpha_1 + \alpha_2 + \alpha_3 \ar@{-}[rd] & \\
                \alpha_1 + \alpha_2 \ar@{-}[rd] & & \alpha_2 + \alpha_3 \\
                \alpha_1 & \alpha_2 & \alpha_3 \\ }
    \end{equation*}
    where roots in the same string are joined by a line. As examples of Peterson translation, we have
    \begin{align*}
        \tau\left(\{\alpha_1, \alpha_1 + \alpha_2, \alpha_1 + \alpha_2 + \alpha_3\}, \alpha_1\right)
            & = \{\alpha_1, \alpha_2, \alpha_2+\alpha_3\}, \\
        \tau\left(\{\alpha_1, \alpha_1 + \alpha_2, \alpha_2 + \alpha_3\}, \alpha_1\right)
            & = \{\alpha_1, \alpha_2, \alpha_2+\alpha_3\}, \\
        \tau\left(\{\alpha_1, \alpha_1 + \alpha_2, \alpha_1 + \alpha_2 + \alpha_3, \alpha_2+\alpha_3\}, \alpha_1\right)
            & = \{\alpha_1, \alpha_2, \alpha_1 +\alpha_2+\alpha_3, \alpha_2+\alpha_3\}, \text{ and } \\
        \tau\left(\{\alpha_1,\alpha_2,\alpha_3\}, \alpha_1\right) &= \{\alpha_1,\alpha_2,\alpha_3\}.
    \end{align*}
\end{example}

\begin{prop}\label{P:petersonfree}
    Let $S$ be a coconvex set in $R^+$. Then:
    \begin{enumerate}[(a)]
        \item The Peterson translation $\tau(S, \alpha)$ is coconvex for every
            $\alpha \in R^+$.
        \item Suppose $\alpha \in R^+$, where $\alpha$ is either a long root,
            or is not contained in a factor of type $C$ or $F$. If $\mcA(S)$ is
            free then $\mcA(\tau(S,\alpha))$ is free with the same coexponents
            as $\mcA(S)$. 
        \item If $S$ is not a lower order ideal, then there is $\alpha \in S$
            such that $\tau(S,\alpha)$ is not equal to $S$.
    \end{enumerate}
\end{prop}
The proof of Proposition \ref{P:petersonfree} is given later in this section.
If $S$ is coconvex and $\alpha \not\in S$ then 
$\tau(S, \alpha) = S$, so we focus on translations by roots $\alpha \in S$. 
\begin{example}
    If $S$ is not coconvex then part (b) of Proposition \ref{P:petersonfree} does not hold,
    even if $\alpha \in S$.  Let $S_0 = \{\alpha_2, \alpha_3, \alpha_1 +
    \alpha_2, \alpha_1+\alpha_2+\alpha_3\} \subseteq A_3$. Then
    \begin{equation*}
        S_1 := \tau(S_0,\alpha_2) = \{\alpha_1, \alpha_2,\alpha_3,\alpha_1+\alpha_2+\alpha_3\},
    \end{equation*}
    so
    \begin{equation*}
         G(S_0) = \vcenter{\xymatrix{ 1 \ar@{-}[r] \ar@{-}[rd] & 4 \ar@{-}[d] \\
                            2 \ar@{-}[r] & 3 \\}} 
        \ \text{ while }\ 
        G(S_1) =  \vcenter{\xymatrix{ 1 \ar@{-}[r] \ar@{-}[d] & 4 \ar@{-}[d] \\
                            2 \ar@{-}[r] & 3 \\}}.
    \end{equation*}
    Thus $\mcA(S_0)$ is free, but $\mcA(S_1)$ is not free.
\end{example}
To prove Proposition \ref{P:petersonfree}, we need some facts about multi-restriction
of arrangements. Given an arrangement $\mcA$ in $V^*$, the restriction $\mcA^H$
to a hyperplane $H \in \mcA$ is the arrangement $\{K \cap H : K \in \mcA
\setminus \mcA_H\}$ in $H$.  A multi-arrangement is an arrangement $\mcA$,
along with a multiplicity function $m : \mcA \arr \Z_{\geq 1}$ which assigns a
multiplicity $m(H)$ to every hyperplane $H \in \mcA$. Equivalently, $(\mcA, m)$
can be regarded as the (non-reduced) scheme cut out by the polynomial 
\begin{equation*}
    Q= \prod_{H \in \mcA} \alpha_H^{m(H)},
\end{equation*}
where $\alpha_H$ is a defining form for $H \in \mcA$.  Given an arrangement
$\mcA$, we let $\widetilde{\mcA}$ denote the multi-arrangement where every
hyperplane in $\mcA$ has multiplicity one. The Zeigler multi-restriction
$\widetilde{\mcA}^H$ is the multi-arrangement with underlying arrangement
$\mcA^H$, and multiplicity function $m(K) = |\{K' \in \mcA : K' \cap H = K\}|$.

Let $S^* V$ denote the space of polynomials on $V^*$. The module of derivations
$\Der(\mcA,m)$ of the multi-arrangement is the space of derivations of $S^* V$
which preserve the defining ideal generated by the polynomial $Q$. For an
ordinary arrangement, the module of derivations $\Der(\mcA)$ is just
$\Der(\widetilde{\mcA})$. A multi-arrangement $(\mcA,m)$ is said to be free if
$\Der(\mcA,m)$ is a free $S^* V$-module.  Zeigler showed that if $\mcA$ is
free, then $\widetilde{A}^{H}$ is free for any $H \in \mcA$ \cite{Zi89}. Abe
and Yoshinaga prove the following converse to Zeigler's theorem:
\begin{thm}[\cite{AY13}, Theorem 4.1]\label{T:freetranslate}
    Let $\mcA$ be an arrangement such that $\widetilde{\mcA}^H$ is free for
    some $H \in \mcA$. Then $\mcA$ is free if and only if $\mcA_X$ is free for
    every flat $X \subseteq H$ of corank $3$.
\end{thm}
If $v_i$ is a basis for $V$, then the Euler derivation $\sum v_i
\frac{\partial}{\partial v_i}$ always lies in $\Der(\mcA)$, and hence if $\mcA$
is free, then $1$ appears as an coexponent with positive multiplicity. If we
write the coexponents of $\mcA$ as $1,m_2,\ldots,m_l$, then Ziegler showed that
the coexponents of $\widetilde{\mcA}^H$ are $m_2,\ldots,m_l$. This leads to:
\begin{cor}[\cite{AY13}, Corollary 4.3]\label{C:freetranslate}
    Suppose $\mcA_1$ and $\mcA_2$ are hyperplane arrangements, and let $H_1$
    and $H_2$ be hyperplanes in $\mcA_1$ and $\mcA_2$ respectively. Suppose
    $\widetilde{\mcA_1}^{H_1} \iso \widetilde{\mcA_2}^{H_2}$, and that $\mcA_1$ is
    free. Then the following are equivalent:
    \begin{enumerate}[(a)]
        \item $\mcA_2$ is free, with the same coexponents as $\mcA_1$.
        \item $(\mcA_2)_X$ is free for every flat $X \subseteq H_2$ of corank $3$.
    \end{enumerate}
\end{cor}

The key lemma we need to prove Proposition \ref{P:petersonfree} is the following:
\begin{lemma}\label{L:petersonrestrict}
    Let $S_0 \subseteq R^+$, and $H = \ker \alpha$ for some $\alpha \in S$. 
    Let $S_1 = \tau(S_0,\alpha)$.
    Then
    \begin{equation*}
        \widetilde{\mcA(S_0)}^H \iso \widetilde{\mcA(S_1)}^{H}.
    \end{equation*}
\end{lemma}
\begin{proof}
    The elements of $S_1$ are $\alpha$-translates of the elements of
    $S_0$, occurring with the correct multiplicities.
\end{proof}

\begin{proof}[Proof of Proposition \ref{P:petersonfree}]
    To start, we observe that parts (a) and (b) of Proposition \ref{P:petersonfree}
    hold for all root systems of rank $\leq 3$. Indeed, it is enough to check
    $R = A_3$, $B_3$, $C_3$, and $G_2$, which we can do on a computer, using the
    methodology described in Section \ref{S:proofs} to check freeness. 

    Now consider a general root system $R$, and let $S$ be a coconvex subset of
    $R^+$. Recall that if $U$ is a subspace of $V$, then $S_U = S \cap U$. The
    key idea of the proof is that Peterson translation is local, in the sense
    that if $\alpha \in U$ then $\tau(S, \alpha)_U = \tau(S_U, \alpha)$.
    Suppose $\beta, \gamma \in R^+ \setminus \tau(S,\alpha)$ such that $\beta +
    \gamma \in R^+$.  Let $U$ be the subspace spanned by $\alpha$, $\beta$, and
    $\gamma$.  Then $S_U$ is a coconvex subset of $R^+_U$, where $R_U$ is a
    root system of rank $\leq 3$, so $\tau(S_U, \alpha) = \tau(S, \alpha)_U$ is
    coconvex.  Hence $\beta + \gamma \not\in \tau(S,\alpha)_U$. It follows that
    $\tau(S,\alpha)$ is coconvex. 

    Similarly suppose that $\mcA(S)$ is free, and let $X$ be a flat of
    $\mcA(\tau(S,\alpha))$ of corank $3$ contained in $H = \ker \alpha$. Let $U
    \subseteq V$ be the subspace of linear forms that vanish on $X$. Then $U$ has
    dimension $3$, and $\alpha \in U$. Furthermore, if $\alpha$ is long in 
    $R$, then $\alpha$ is long in $R_U$. Since $\mcA(S)$ is free, $\mcA(S_U)$ is
    also free, and thus $\mcA(\tau(S_U, \alpha)) = \mcA(\tau(S,\alpha)_U) = 
    \mcA(\tau(S,\alpha))_X$ is free. Applying Lemma \ref{L:petersonrestrict}
    and Corollary \ref{C:freetranslate}, part (b) we get that
    $\mcA(\tau(S,\alpha))$ is free with the same coexponents as $\mcA(S)$. 

    Finally, suppose that $S$ is a coconvex set but is not a lower order ideal.
    This means that there is $\beta \in S$ and $\gamma \in R^+$ such that
    $\beta - \gamma \not\in S$. Since $S$ is coconvex, $\gamma$ must be in $S$,
    and $\tau(S,\gamma)$ will include some element $\beta - c \gamma \in R^+$,
    $c>0$, which is not in $S$.
\end{proof}

If $\tau(S,\alpha) \neq S$, then the sum of the heights of roots in
$\tau(S,\alpha)$ is strictly less than the sum of the heights of roots in $S$.
Given a coconvex set $S_0$ which is not an order ideal, Proposition
\ref{P:petersonfree} implies that we can translate by $\alpha_0 \in S_0$ to get
a different coconvex set $S_1 = \tau(S_0, \alpha_0)$. Repeating this procedure,
we must eventually arrive at a lower order ideal, leading to the following
corollary of Theorem \ref{T:idealfree} and Proposition \ref{P:petersonfree}.
\begin{cor}\label{C:petersonfree}
    If $S$ is a coconvex set, then there is a sequence $S = S_0, \ldots, S_r$
    of distinct coconvex sets, such that $S_{i+1} = \tau(S_i,\alpha_i)$ for
    some $\alpha_i \in S_i$, and $S_r$ is a lower order ideal. If $\mcA(S)$ is
    free, and $R$ has no factors of type $C$ or $F$, then $\mcA(S)$ has
    coexponents $\Exp(S_r)$.
\end{cor}
In the next section we will show that we can remove the type restriction in
Corollary \ref{C:petersonfree} if $S$ is an inversion set. 

\begin{example}\label{Ex:typeAreflection}
    Using the notation of Example \ref{Ex:typeA}, let $w = s_1 s_2 s_3 s_2 s_1$
    in $W(A_3)$. Then 
    \begin{equation*}
        I(w) = \{\alpha_1, \alpha_3, \alpha_1+\alpha_2, \alpha_2+\alpha_3, 
            \alpha_1 + \alpha_2 + \alpha_3\} = R^+ \setminus \{\alpha_2\}.
    \end{equation*}
    Since $w \neq s_2 s_1 s_3 s_2$, $\mcA(I(w))$ is free (alternatively, the
    associated graph is the complete graph minus an edge). We have
    \begin{align*}
        \tau\left(I(w), \alpha_1\right) & = \{\alpha_1, \alpha_2, \alpha_3, 
            \alpha_2 + \alpha_3, \alpha_1 + \alpha_2 + \alpha_3\} =: S_1, \text{ and } \\
        \tau\left(S_1, \alpha_3\right) & = \{\alpha_1, \alpha_2, \alpha_3, \alpha_1+\alpha_2,
            \alpha_2+\alpha_3\} =: S_2.
    \end{align*}
    The set $S_2$ is an order ideal, and $\Exp(S_2) = \{1, 2, 2\}$, so
    $\mcA(I(w))$ has coexponents $\{1,2,2\}$. 
\end{example}

\begin{example}\label{Ex:matroidiso}
    Although Peterson translation preserves freeness in type $A$, it is not a
    matroid invariant. Let
    \begin{equation*}
        S_0 = \{\alpha_2, \alpha_3, \alpha_4, \alpha_1+\alpha_2, \alpha_2+\alpha_3,
            \alpha_3+\alpha_4, \alpha_1+\alpha_2+\alpha_3\} \subseteq A_4^+.
    \end{equation*}
    The set $S_0$ is coconvex, and the graph of $S_0$ is
    \begin{equation*}
        G(S_0) = \vcenter
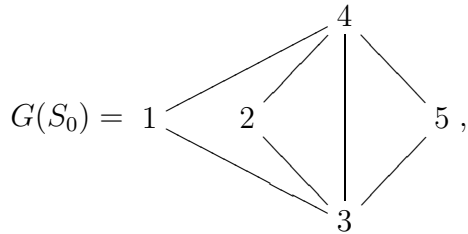
{ \xymatrix{ & & 4 \ar@{-}[dd] \ar@{-}[rd] & \\
                                    1 \ar@{-}[rru] \ar@{-}[rrd] & 2 \ar@{-}[ru] \ar@{-}[rd] & & 5 \\
                                    & & 3 \ar@{-}[ru] & }},
    \end{equation*}
    so $\mcA(S_0)$ is free. Now
    \begin{equation*}
        S_1 := \tau(S_0, \alpha_2) = \{\alpha_1, \alpha_2, \alpha_3, \alpha_4, \alpha_2+\alpha_3, \alpha_3
                + \alpha_4, \alpha_1+\alpha_2+\alpha_3\},
    \end{equation*}
    and
    \begin{equation*}
        G(S_1) = \vcenter
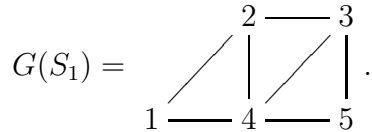
{ \xymatrix{ & 2 \ar@{-}[r] \ar@{-}[d] & 3 \ar@{-}[d] \\
                                    1 \ar@{-}[ru] \ar@{-}[r] & 4 \ar@{-}[ru] \ar@{-}[r] & 5 }}.
    \end{equation*} 
    The vector matroids of $S_0$ and $S_1$ are the cycle matroids of $G(S_0)$
    and $G(S_1)$ respectively. The graph $G(S_1)$ has a simple $5$-cycle, while
    $G(S_0)$ does not, so the cycle matroids cannot be isomorphic. However, $\mcA(S_1)$ is still free.
\end{example}

\section{Peterson-freeness and the Peterson translation graph}\label{S:ptgraph}

In this section we further analyze Peterson translation in types $C$ and $F$. We
start by recalling the definition of Peterson-freeness from the introduction.
\begin{defn}
    We say that a coconvex set $S$ is \emph{Peterson-free} if $\mcA(S)$ is free, and
    $\mcA(\tau(S,\alpha))$ is free for any $\alpha \in S$.
\end{defn}
As in the previous section, if $S$ is Peterson-free then $\mcA(\tau(S,\alpha))$
automatically has the same exponents as $\mcA(S)$ for all $\alpha \in R^+$. 

\begin{prop}\label{P:petersonfree2}
    Let $S$ be a coconvex set such that $\mcA(S)$ is free. Then the
    following are equivalent:
    \begin{enumerate}[(a)]
        \item $S$ is Peterson-free.
        \item $S_U$ is Peterson-free in $R_U$ for any subspace $U \subseteq V$.
        \item $S_U$ is Peterson-free for every subspace $U\subseteq V$ of dimension $3$. 
    \end{enumerate}
\end{prop}
\begin{proof}
    To show that (a) implies (b), suppose that $S_U$ is not Peterson-free for
    some subspace $U$. Then there is $\alpha \in S_U$ such that
    $\mcA(\tau(S_U,\alpha)) = \mcA(\tau(S,\alpha)_U)$ is not free. It follows
    that $\mcA(\tau(S,\alpha))$ is not free, and hence $S$ is not
    Peterson-free. 

    It is clear that (b) implies (c). To finish the proof, we show that
    (c) implies (a). Suppose $S$ is not Peterson-free. Since $\mcA(S)$ is free,
    there is $\alpha \in S$ such that $\mcA(\tau(S,\alpha))$ is not free. By
    Lemma \ref{L:petersonrestrict} and Corollary \ref{C:freetranslate}, there
    must be a flat $X$ of corank $3$ contained in $\ker \alpha$ such that
    $\mcA(\tau(S,\alpha))_X$ is not free.  If $U$ is the space of linear forms
    vanishing on $X$, then $\mcA(S_U)$ is free and $\mcA(\tau(S_U,\alpha)) =
    \mcA(\tau(S,\alpha))_X$ is not. Thus $S_U$ is not Peterson-free. 
\end{proof}
By Proposition \ref{P:petersonfree}, Peterson-freeness is equivalent to
freeness if $R$ has no factors of type $C$ or $F$. Thus in part (c) of
Proposition \ref{P:petersonfree2} it suffices to check subspaces $U$ such that
$R_U \iso C_3$. The positive roots of $C_3$, arranged in order of height, are
\begin{gather*}
    \alpha_1 + 2 \alpha_2 + 2 \alpha_3 \\
    \alpha_1 + 2 \alpha_2 + \alpha_3 \\
    \alpha_1 + 2 \alpha_2 \quad\quad \alpha_1 + \alpha_2 + \alpha_3 \\
    \alpha_1 + \alpha_2 \quad\quad\quad \alpha_2 + \alpha_3 \\
    \alpha_1 \quad\quad\quad \alpha_2 \quad\quad\quad \alpha_3
\end{gather*}
where $\alpha_1$ is a long simple root, and $\alpha_2$ and $\alpha_3$ are
short simple roots. Using a computer, we can easily
check that the only coconvex sets $S$ in $C_3$ for which $\mcA(S)$ is free
but $S$ is not Peterson-free are
\begin{align}
    S_1 & = \{ \alpha_1, \alpha_2, \alpha_3, \alpha_2 + \alpha_3, \alpha_1 + 2\alpha_2, \alpha_1 + 2\alpha_2 + 2\alpha_3\}, \nonumber \\ 
    S_2 & = \{ \alpha_1, \alpha_3, \alpha_1 + \alpha_2 , \alpha_1 + 2\alpha_2, \alpha_1 + \alpha_2 + \alpha_3, \alpha_1 + 2\alpha_2 + 2\alpha_3 \}, \text{ and } \label{E:nonpf} \\
    S_3 & = \{ \alpha_2, \alpha_3, \alpha_1 + \alpha_2, \alpha_2 + \alpha_3, \alpha_1 + \alpha_2 + \alpha_3, \alpha_1 + 2\alpha_2 + \alpha_3\}.  \nonumber
\end{align}
If we add these three patterns to the $50$ minimal non-free coconvex patterns
in type $C_3$, then we get a pattern avoidance characterization of
Peterson-freeness, rather than freeness. Also, none of these sets are biconvex,
so we immediately have the following:
\begin{cor}\label{C:ptinversion}
    If $S \subset R^+$ is an inversion set and $\mcA(S)$ is free, then
    $S$ is Peterson-free.
\end{cor}

The most important property of Peterson-freeness is the following:
\begin{prop}\label{P:petersonfree3}
    If $S \subset R^+$ is Peterson-free and $\alpha \in R^+$ then $\tau(S,
    \alpha)$ is also Peterson-free. 
\end{prop}
\begin{proof}
    Suppose $\mcA(S)$ is free, but $\tau(S,\alpha)$ is not Peterson-free for
    some $\alpha \in S$. We want to show that $S$ is not Peterson-free. 
    We can assume that $\mcA(\tau(S,\alpha))$ is free, since otherwise
    $S$ is not Peterson-free by definition. With this assumption, Proposition
    \ref{P:petersonfree2} implies that there is a subspace $U \subseteq V$
    of dimension $3$ for which $\tau(S,\alpha)_U$ is not Peterson-free.
    Let $U' = U + \R \alpha$. Then $\tau(S,\alpha)_{U'} = \tau(S_{U'}, \alpha)$
    is not Peterson-free by Proposition \ref{P:petersonfree2} again, and we
    will be done if we can show that $S_{U'}$ is not Peterson-free. Since
    $U'$ has dimension $\leq 4$, it suffices to check Proposition
    \ref{P:petersonfree} for root systems of rank $\leq 4$. Since freeness and
    Peterson-freeness are equivalent except in types $C$ and $F$, we finish the
    proof by checking on a computer that the proposition holds in $C_4$ and
    $F_4$. 
\end{proof}

It is helpful to think of Peterson translation and Peterson-freeness using
a certain directed graph:
\begin{defn}
    Let $R$ be a finite root system. The \emph{Peterson translation graph}
    of $R$ is the directed graph $\mcG \cong \mcG(R)$ with vertex set
    \begin{equation*}
        V(\mcG) = \{ S \subseteq R^+\ :\ S \text{ is coconvex} \},
    \end{equation*}
    and an edge $S_0 \arr S_1$ if $S_1 = \tau(S_0, \alpha) \neq S_0$
    for some $\alpha \in S_0$. 
\end{defn}
The graph $\mcG$ is acyclic, and by Proposition
\ref{P:petersonfree}, part (c) the terminal vertices of $\mcG$ are the lower
order ideals in $R^+$. In terms of the Peterson translation graph, Proposition
\ref{P:petersonfree3} is equivalent to:
\begin{cor}\label{C:petersonfree2}
    A coconvex set $S \subset R^+$ is Peterson-free if and only if $\mcA(T)$ is
    free for every coconvex set $T$ for which there is a directed path $S \leadsto T$ in $\mcG(R)$.
\end{cor}
In particular, if $\mcA(S)$ is free and $S$ is an inversion arrangement then
we can translate $S$ to a lower order ideal $T$, and $\mcA(S)$ will have
coexponents $\Exp(T)$ regardless of type. 
\begin{rmk}
    It might be interesting to study the Peterson translation graph from a combinatorial
    perspective. For instance, it is possible to reach every lower order ideal
    by Peterson translation from an inversion set?
\end{rmk}

\section{Proof of pattern avoidance results}\label{S:proofs}

In this section, we prove the main pattern avoidance results from Section
\ref{S:pattern}. In subsection \ref{SS:computer} we also discuss our
methodology for checking the many statements left up to computer verification
in this and previous sections. We start by restricting our attention to those
edges of $\mcG(R)$ which preserve freeness locally.
\begin{defn}
    Given a root system $R$, let $\mcG^{Fr}(R)$ be the subgraph of the Peterson
    translation graph $\mcG(R)$ with the same vertex set, but only those edges $S
    \arr \tau(S, \alpha)$ for which there does not exist a subspace $\alpha \in U
    \subset V$ such that $\mcA(S_U)$ is free but $\mcA(\tau(S_U,\alpha))$ is not. 
\end{defn}
By Lemma \ref{L:petersonrestrict} and Corollary
\ref{C:freetranslate}, to test whether $S \arr
\tau(S,\alpha)$ is an edge in $\mcG^{Fr}(R)$ it
is enough to check subspaces $\alpha \in U
\subset V$ of dimension $3$.  If $R$ does not
contain any factors of types $C$ or $F$ then
$\mcG^{Fr}(R) = \mcG(R)$. 

We now can define two properties of a root system $R$ and fixed integer
$k$:
\begin{equation*}
    (L_k) \quad \parbox{5in}{
            If $S \subseteq R^+$ is coconvex, and $\mcA(S)_X$ is free for all
            flats $X$ of corank $\leq k$, then $\mcA(S)$ is free.}
\end{equation*}
\begin{equation*}
    (T_k) \quad \parbox{5in}{
            If $S$ is a terminal vertex of $\mcG^{Fr}(R)$, and $\mcA(S)_X$ is free
            for all flats $X$ of corank $\leq k$, then $\mcA(S)$ is free.}
\end{equation*}
The main result of this section is a criterion for $(L_k)$ to hold for a class
of root systems, assuming that $(T_k)$ holds for all elements of the class.  
\begin{prop}\label{P:Fk}
    Let $\mcC$ be a class of finite root systems with the property that if $R
    \in \mcC$, and $U$ is a subspace of the ambient space of $R$, then $R_U$ is
    isomorphic to an element of $\mcC$. Suppose there is some $k \geq 3$ such
    that $(T_k)$ holds for all irreducible root systems in $\mcC$, and 
    $(L_k)$ holds for all irreducible root systems in $\mcC$ of rank $k+1$. 
    Then $(L_k)$ holds for all $R \in \mcC$.
\end{prop}
\begin{proof}
    First note that our hypothesis implies that $(L_k)$ holds for all root
    systems in $\mcC$ of rank $\leq k+1$.  Let $R \in \mcC$, and suppose $S
    \subseteq R^+$ is a coconvex set with the property that $\mcA(S)_X$ is free
    for all flats $X$ of corank $\leq k$.  This is equivalent to saying that
    $\mcA(S_U)$ is free for all subspaces $U \subseteq V$ spanned by at most $k$
    elements of $S$. 

    We show that $\mcA(S)$ is free by induction on the sum of
    the heights of the roots in $S$. If $S$ is a terminal vertex in $\mcG^{Fr}(R)$,
    then $\mcA(S)$ is free by $(T_k)$. Otherwise, there is $\alpha \in S$ such
    that $S \arr \tau(S,\alpha)$ is an edge in $\mcG^{Fr}$. Now suppose that
    $U$ is a subspace of $V$ spanned by at most $k$ elements of
    $\tau(S,\alpha)$, and let $U' = U + \R \alpha$.
    Then $S_{U'}$ is a coconvex subset of $R_{U'}^+$, and $R_{U'}$ is a root
    system of rank $\leq k+1$. If $U''$ is spanned by $\leq k$ elements of
    $S_{U'}$, then $\mcA\left((S_{U'})_{U''}\right) = \mcA(S_{U''})$ is free,
    and since $(L_k)$ holds for $R_{U'}$ we conclude that $\mcA(S_{U'})$ is
    free. Since $\alpha \in U'$, $\mcA(\tau(S,\alpha)_{U'}) =
    \mcA(\tau(S_{U'},\alpha))$, and $\mcA(\tau(S_{U'}, \alpha))$ is free by the definition of $\mcG^{Fr}$. Hence
    $\mcA(\tau(S,\alpha)_U)$ is free for all subspaces $U$ spanned by at most
    $k$ elements of $\tau(S,\alpha)$. Since the sum of the heights of the roots
    in $\tau(S,\alpha)$ is less than the sum of the heights of the roots in
    $S$, we conclude by induction that $\mcA(\tau(S,\alpha))$ is free. But
    since $k \geq 3$, we know that $\mcA(S_U)$ is free for all subspaces $U$ of
    dimension $\leq 3$, and hence $\mcA(S)$ is free by Lemma
    \ref{L:petersonrestrict} and Corollary \ref{C:freetranslate}, part (b).
\end{proof}

\begin{proof}[Proof of Theorem \ref{T:coconvex} and Corollary \ref{C:coconvexpat}]
    Clearly we only need to prove the results for irreducible root systems. 
    For $F_4$ and $G_2$ there is nothing to be done, since these root systems
    have rank $\leq 4$. The root subsystems of an irreducible root system $R$
    are well-known; see for instance \cite[Tables 9 and 10]{Dyn57}, as well as
    \cite{DL11} for a modern account. The root systems $R_U$, for $U$ a subspace
    of $V$, can be easily determined from these results. For instance,
    the restriction $R_U$ of $R = B_n$ to a
subspace is always a direct sum of root systems
of type $A$ and $B$. Thus if $\mcC$ is the class
of root systems $R$ whose irreducible factors are
of types $A$ or $B$, then the restriction $R_U$
    of a rootsystem $R \in \mcC$ to a subspace $U$ is also in $\mcC$. By Proposition
    \ref{P:petersonfree}, $\mcG^{Fr}(R) = \mcG(R)$ for every $R \in \mcC$, and
    the terminal elements of $\mcG^{Fr}(R)$ are simply the lower order
    ideals in $R$, so the condition $(T_k)$ holds
    for every $R \in \mcC$ (irregardless of $k$). We use a computer to check that
    $(L_3)$ holds for $A_4$ and $B_4$, after which Proposition \ref{P:Fk}
    implies that $(L_3)$ holds for all root systems in $\mcC$.

    If $\mcC$ is instead the class of simply-laced root systems (those whose
    irreducible factors are of types $A$, $D$, and $E$), we can make the
    exact same argument, except that $(L_3)$ does not hold for $D_4$. 
    However, we can verify by computer that $(L_4)$ holds for $D_5$.
    Since we already know that $(L_3)$ holds for every $A_n$, the condition
    $(L_4)$ holds for $A_5$, and Proposition \ref{P:Fk} implies that
    $(L_4)$ holds for all simply-laced root systems. 

    This leaves type $C_n$, which is the only classical type for which $(T_k)$
    does not hold trivially. Let $\mcC$ be the class of root systems whose
    irreducible factors are of types $A$ or $C$. This class is also closed under restriction to a subspace. We can verify on a computer
    that $(L_3)$ holds for $C_4$. To finish the argument, we show that 
    $(T_3)$ holds for all $C_n$, $n \geq 4$ (from which it follows immediately
    that $(T_3)$ holds for all elements of $\mcC$).

    We use the usual presentation of the root system $R = C_n$, with
    \begin{equation*}
        C_n^+ = \{e_j - e_i : 1 \leq i < j \leq n \} \cup \{ e_i + e_j : 1 \leq i \leq j \leq n\}.
    \end{equation*}
    With this choice of positive roots, the simple roots are $\alpha_1 = 2e_1$ and $\alpha_i = e_i - e_{i-1}$,
    $i=2,\ldots,n$. The root system $A_{n-1}$ is canonically contained in $C_n$
    as the span of the simple roots $\alpha_2, \ldots, \alpha_n$, or
    equivalently as the set of vectors of the form $e_j - e_i$, $i \neq j$.
    If $k \leq n$, we regard $C_k$ and $A_{k-1}$ as natural subsystems of $C_n$
    and $A_{n-1}$. We need three lemmas:

    \begin{lemma}\label{L:prooflem1}
        A root $\alpha \in C_n^+$ belongs to the subsystem $A_{n-1}^+$ if and
        only if there is a long root $\beta \in C_n^+$ such that if $U_0 =
        \vspan\{\alpha, \beta\}$, then $R_{U_0} \iso C_2$ and $\alpha$ and
        $\beta$ are the short and long simple roots of $R_{U_0}$ respectively.
        Furthermore,
        if $\alpha \in A_{n-1}^+$ then the root $\beta$ and subspace $U_0$
        are unique.
    \end{lemma}
    \begin{proof}
        If $\alpha$ is short but not in $A_{n-1}$, then $\alpha$ must be
        of the form $e_i + e_j$ for $i \neq j$, and there is no positive long
        root $\beta$ such that $\beta + \alpha \in R^+$. So if there is a 
        subspace $U_0 \ni \alpha$ such that $R_{U_0} \iso C_2$ and $\alpha$
        is the short simple root in $R_{U_0}$, then we must have $\alpha \in A_{n-1}$. 
        
        Conversely, if $\alpha = e_j - e_i$ for $i < j$ then there is a 
        unique $\beta$ such that $R_{U_0} \iso C_2$, namely $\beta = 2 e_i$. 
    \end{proof}

    \begin{lemma}\label{L:prooflem2}
        Suppose $S$ is a terminal vertex in $\mcG^{Fr}(C_n)$, and $\alpha$
        is a root in $S$ such that $\tau(S, \alpha) \neq S$. Then $\alpha$
        belongs to $A_{n-1}$ and, in the notation from Lemma \ref{L:prooflem1},
        one of the following conditions holds:
        \begin{enumerate}[(a)]
            \item $S_{U_0} = \{ \alpha, \beta, \beta + 2 \alpha\}$, or
            \item $S_{U_0} = \{ \alpha, \beta + \alpha\}$. 
        \end{enumerate}
    \end{lemma}
    \begin{proof}
        By the definition of $\mcG^{Fr}$, there must be a subspace $U$ containing
        $\alpha$ such that $\mcA(S_U)$ is free and $\mcA(\tau(S_U, \alpha))$ is 
        not. By Corollary \ref{C:freetranslate} and Lemma \ref{L:petersonrestrict}, we
        can take $U$ to be of dimension $3$. As discussed in Section \ref{S:ptgraph},
        we must have $R_U \iso C_3$, and this isomorphism must send $S_U$ to
        one of the coconvex sets $S_i \subset C_3$ from Equation (\ref{E:nonpf}). 
        If $S_U$ is sent to $S_2$, then there is a long root $\beta$ such that
        $\tau(S_U, \beta) \neq S_U$. Consequently $\tau(S, \beta) \neq S$, and
        by Proposition \ref{P:petersonfree}, $S$ is not terminal. So $S_U$
        is sent to either $S_1$ or $S_3$. The only roots $\alpha'$ in $S_1$ and
        $S_3$ for which $\tau(S_i, \alpha')$ is not free are $\alpha_2$,
        $\alpha_3$, and $\alpha_2 + \alpha_3$, so $\alpha$ must be sent to one
        of these roots in $C_3$. Since all these roots belong to $A_{2} \subseteq C_3$, it
        follows from Lemma \ref{L:prooflem1} that there is $\beta' \in S_U$
        such that $\beta'$ and $\alpha$ form the simple roots of a 
        subsystem $R_{U_0} \iso C_2$. Thus $\alpha$ lies in $A_{n-1}$, 
        and the rest of the lemma follows from looking at the image of $S_{U_0}$
        inside of $S_1$ and $S_3$. 
    \end{proof}

    \begin{lemma}\label{L:prooflem3}
        Suppose $S$ is a terminal vertex in $\mcG^{Fr}(C_n)$. 
        \begin{itemize}
            \item If $e_k - e_i \in S$ for $1 \leq i < k$ then $e_j - e_i \in S$
                for all $i < j < k$.
    
            \item If $e_l + e_k \in S$ for some $1 \leq l, k \leq n$, then $e_j
                    - e_i \in S$ for all $1 \leq i < j \leq k$.
        \end{itemize}
    \end{lemma}
    \begin{proof}
        We start by showing that if $e_l + e_k \in S$, then $e_k - e_i \in S$
        for all $1 \leq i < k$. Indeed, if $e_i + e_l \not\in S$ then $e_k - e_i =
        (e_l+e_k) - (e_i + e_l) \in S$ by coconvexity. If $e_i + e_l \in S$
        then $\tau(S, e_i + e_l) = S$ by Lemma \ref{L:prooflem2}, and we get
        the same conclusion.

        Now for the first part of the lemma, let $\alpha = e_k - e_j$, so
        $e_j - e_i = (e_k - e_i) - \alpha$. If $\alpha \not\in S$ or $\tau(S,
        \alpha) = S$ then we are done, so suppose $\alpha \in S$ and
        $\tau(S, \alpha) \neq S$. Let $\beta = 2 e_j$ and $U_0 = \vspan\{
        \alpha, \beta\}$, so $R_{U_0} \iso C_2$ with simple roots 
        $\alpha$ and $\beta$. If condition (a) holds in Lemma \ref{L:prooflem2}
        then $\beta \in S_{U_0}$, while if condition (b) holds then
        $\beta + \alpha = e_k + e_j \in S_{U_0}$. In both cases
        $e_j - e_i \in S$ by first paragraph.

        For the second part of the lemma, we know that $e_k - e_i
        \in S$ for $1 \leq i < k$, and hence $e_j - e_i \in S$ for
        all $1 \leq i < j \leq k$ by the first part of the lemma.
    \end{proof}
    Suppose $S$ is a terminal vertex of $\mcG^{Fr}(C_n)$.  We can assume that
    $S$ is not a lower order ideal, since otherwise $\mcA(S)$ is free by
    Theorem \ref{T:idealfree}. We can also assume that $S \not\subseteq
    C_{n-1}$, or in other words that $S \setminus C^+_{n-1}$ is non-empty. We
    first consider the case that $S \setminus C^+_{n-1}$ is contained in
    $A^+_{n-1}$. If $e_n - e_i$, $e_n - e_j \in S \setminus
    C^+_{n-1}$, where $1 \leq i < j < n$, then $e_j - e_i  = (e_n - e_i)
    - (e_n - e_j) \in S$ by Lemma \ref{L:prooflem3}. Consequently 
    \begin{equation*}
        X = \bigcap_{\alpha \in S \cap C_{n-1}^+} \ker \alpha
    \end{equation*}
    is a modular coatom of $\mcA(S)$. Thus $\mcA(S)$ will be free if and only
    if $\mcA(S)_X = \mcA(S \cap C_{n-1}^+)$ is free. 

    Now suppose that $S \setminus C^+_{n-1}$ is not completely contained
    in $A^+_{n-1}$. Then $e_l + e_n \in S$ for some $1 \leq l \leq n$, and
    hence $e_j - e_i \in S$ for all $1 \leq i < j \leq n$ by Lemma
    \ref{L:prooflem3}. The free arrangements $\mcA(S)$ where $S$ contains
    $A^+_{n-1}$ have been characterized by Edelman and Reiner \cite[Theorem
    4.6]{ER94}. In particular, they show that $\mcA(S)$ is free if $\mcA(S)_X$
    is free for all flats $X$ of corank $\leq 4$.\footnote{In fact, their
    criterion applies even when $S$ is not coconvex. To see that freeness 
    follows from checking flats of corank $\leq 4$, use the first part of the
    proof of Theorem 4.6 combined with Theorem 4.1 and Lemma 4.5.} But since
    we already know that $(L_3)$ holds for $C_4$, it is enough to check that
    $\mcA(S)_X$ is free for all flats of corank $\leq 3$. Thus $(T_3)$
    holds for all $C_n$, finishing the proof.
\end{proof}

\subsection{Determining freeness on a computer}\label{SS:computer}

Many of the proofs in this paper rely on computer verification. 
The positive roots of a given root system can be listed using
common mathematical software, such as Maple (with John Stembridge's coxeter package)
or the Sage Mathematics system. Given the positive roots, it is straightforward
to list all coconvex sets or compute the Peterson translation of a coconvex
set.  Verifying the freeness of an arrangement $\mcA(S)$ is more complicated, so
in this section we explain how this is done. 

Given $H \in \mcA$, let $\mcA \setminus H$ denote the deletion of $H$ from
$\mcA$. An arrangement is said to be inductively free with exponents
$m_1,\ldots,m_l$ if either $\mcA$ is empty and $m_1=\ldots=m_l=0$, or there is
a hyperplane $H \in \mcA$ and index $i$ such that $\mcA \setminus H$ is
inductively free with exponents $m_1,\ldots,m_{i-1},m_i-1,m_{i+1}, \ldots,m_l$,
and $\mcA^H$ is inductively free with exponents
$m_1,\ldots,m_{i-1},m_{i+1},\ldots,m_{l}$. Inductive freeness can be checked on
a computer in low rank using brute force, and in higher ranks using
optimizations as in \cite{Sl13}.

Given a real arrangement $\mcA$, the Poincare polynomial of $\mcA$ is
\begin{equation*}
    Q(\mcA, t) = \sum t^i \dim H^i(\mcM(\mcA)), 
\end{equation*}
where $\mcM(\mcA)$ is the complement of the complexification of $\mcA$. If
$\mcA$ is free with coexponents $m_1,\ldots,m_l$, then $Q(\mcA,t) = \prod_i
(1+m_i t)$ by a theorem of Terao \cite{Te81}. We say that $Q(\mcA,t)$ splits if
it can be written in this form, for some non-negative integers
$m_1,\ldots,m_l$. The polynomial $Q(\mcA,t)$ can easily be computed using a
restriction-deletion recurrence, and thus we can determine if $Q(\mcA,t)$
splits on a computer.
\begin{defn}
    We say that $\mcA$ is verifiable free if $\mcA$ is inductively free, and
    verifiably non-free if $Q(\mcA,t)$ does not split. We say that freeness of
    $\mcA$ is verifiable if $\mcA$ is verifiably free or non-free. 
\end{defn}
It is not true that the freeness of $\mcA(S)$ is verifiable for every coconvex
set $S$. However, freeness is always verifiable in $A_3$, $B_3$, and $C_3$, and
this suffices for the proof of Proposition \ref{P:petersonfree}.  

To determine the minimal non-free coconvex patterns in a root system, we use
the following procedure for each coconvex set $S$:
\begin{itemize}
    \item Check whether or not $\mcA(S)$ is verifiably free. If $\mcA(S)$
        is verifiably free, stop and return FREE.
    \item Otherwise, for each subspace $U$ of the ambient space $V$ with
        $3 \leq \dim U < \rank \mcA(S)$,
        check whether or not the freeness of $\mcA(S_U)$ is verifiable. 
        \begin{itemize}
            \item If $\mcA(S_U)$ is verifiably non-free, stop and return
                NOT MINIMAL.
            \item If freeness of $\mcA(S_U)$ is not verifiable, stop and
                return AMBIGUOUS. 
            \item Otherwise, continue.
        \end{itemize}
    \item At the end of the loop, check whether or not $\mcA(S)$ is verifiably
        non-free. If so, return MINIMAL PATTERN.
    \item If, at this point, the freeness of $\mcA(S)$ is not verifiable,
        return AMBIGUOUS.
\end{itemize}
This procedure outputs AMBIGUOUS if $\mcA(S)$ is not verifiably
free, and either there is a subspace $U$ such that freeness of $\mcA(S_U)$ is
not verifiable, or $\mcA(S_U)$ is verifiably free for all subspaces $U$ and
$\mcA(S)$ is not verifiably non-free. However,
this procedure does not output AMBIGUOUS for any of the root systems $A_3$, $B_3$, $C_3$, $D_4$,
$F_4$, $A_4$, $B_4$, $C_4$, and $D_5$.  In
particular, we can use this procedure to check
that only the first five of these root systems
have any minimal non-free coconvex patterns, from
which we conclude that $(L_3)$ holds in
$A_4$--$C_4$ and $(L_4)$ holds in $D_5$. 

We use a similar method to check freeness of coconvex sets in the proof of Proposition \ref{P:petersonfree3}. Assuming that we can test for freeness, it is not hard to construct the Peterson translation graph, and from this graph determine all Peterson-free elements in types $C_4$ and $F_4$. 

The programs used to perform the computer
verifications listed in this paper are available
from the author's website. Implementations are
provided in Maple and C++, with the C++
implementation running in a matter of minutes for
$D_5$.  

\section{Geometric interpretation of the Peterson translate}\label{S:geometry}

The combinatorial Peterson translation defined in the previous section comes
from a geometric construction on the flag variety introduced by Peterson and
further developed by Carrell and Kuttler \cite{CK03}. Let $G$ be the linear
algebraic group associated to $R$, and fix a choice of Borel $B$ and maximal
torus $T \subseteq B$ compatible with the choice of positive roots $R^+$. As in
the introduction, let $X = G/B$ be the flag variety associated to $R$, and let
$X(w)$ be the Schubert variety indexed by $w \in W(R)$. Let $x, y \in W(R)$ be
two elements differing by a reflection, so $x = r_{\alpha} y$, where
$r_{\alpha}$ is reflection through some root $\alpha \in R^+$. Suppose further
that $x < y$ in Bruhat order, or equivalently that $\alpha \in I(y)$. There is
a $T$-invariant curve $C \subseteq X$ with $C^T = \{x,y\}$, and any $T$-module $M
\subseteq T_y X$ can be translated along $C$ to a $T$-module
$\tilde{\tau}(M,\alpha) \subseteq T_x X$, called the Peterson translate of $M$
along $C$. This geometric Peterson translate has been used by Carrell and
Kuttler to study smoothness for Schubert varieties, and we refer to \cite{CK03}
for more details. If $x,y \leq w$ in Bruhat order, then $C$ is contained in
$X(w)$, and hence if $M \subseteq T_y X(w)$ then $\tilde{\tau}(M,\alpha) \subseteq
T_x X(w)$.

Given a $T$-module $M$, let $\Omega M$ denote the set of $T$-weights. Given
$y \in W(R)$, we have
\begin{equation*}
    \Omega T_y X = y \Omega T_e X = y R^- = I(y) \cup \{\beta \in R^- : y^{-1} \beta \in R^-\}.
\end{equation*}
Given $S \subseteq \Omega T_y X$, there is a unique $T$-submodule $M \subseteq T_y X$
with $S = \Omega M$, and for any $\alpha \in I(y)$ we can define
\begin{equation*}
    \tilde{\tau}(S,\alpha) = \Omega \tilde{\tau}(M, \alpha) \subseteq \Omega T_{x} X,
\end{equation*}
where $x = r_{\alpha} y$. Carrell and Kuttler
give an explicit formula for
$\tilde{\tau}(S,\alpha)$, and this formula is the
basis for the definition of Peterson translation
in Section \ref{S:peterson}. The relation between the two notions of Peterson translation can be
explicitly stated as follows:
\begin{prop}\label{P:geometricpeterson}
    With notation as above, we have
    \begin{equation}\label{E:geom2comb}
        -x^{-1} \tilde{\tau}(S,\alpha) = \tau(-y^{-1} S, -y^{-1} \alpha),
    \end{equation}
    or in other words, the map $- y^{-1} : \Omega T_y X \arr R^+$ transforms
    $\tilde{\tau}(\cdot,\alpha)$ to combinatorial Peterson translation by
    $-y^{-1}\alpha$. 
\end{prop}
\begin{proof}
    Let $S$ be a subset of $\Omega T_y X$, and let $\alpha \in I(y)$, $x =
    r_{\alpha} y$.  Carrell and Kuttler give a formula for $\tilde{\tau}(S,
    \alpha)$ as follows: Let $\mfg$ be the Lie algebra of $G$, and let
    $\mfg_{\beta}$ be the root spaces of the roots $\beta \in R$. Let $T^{\alpha}$
    be the kernel of the character $e^{\alpha} : T \arr \C$ (the subtorus
    $T^{\alpha}$ is the stabilizer of a generic point in the curve $C$ connecting
    $y$ and $x$). Let $M$ be the submodule of $T_y X$ with $\Omega M = S$. Note
    that since $T_e X$ is a $\mfg_{-y^{-1} \alpha}$-module, $T_y X$ is a
    $\mfg_{-\alpha}$-module. Carrell and Kuttler show that
    \begin{equation}\label{E:geomtranslate}
        \tilde{\tau}(S,\alpha) = r_{\alpha} \Omega M^{-\alpha},
    \end{equation}
    where $M^{-\alpha}$ is the unique $\mfg_{-\alpha}$-submodule of $T_y X$ such
    that $M^{-\alpha}$ is isomorphic to $M$ as a $T^{\alpha}$-module. So
    \begin{equation*}
        - x^{-1} \tilde{\tau}(S,\alpha) = -x^{-1} r_{\alpha} \Omega M^{-\alpha}
            = -y^{-1} \Omega M^{-\alpha}.
    \end{equation*}
    Let $\alpha' = -y^{-1} \alpha \in R^+$. Then $-y^{-1}$ gives a bijective
    correspondence between $\mfg_{-\alpha}$-submodules of $T_y X$ and
    $\mfg_{-\alpha'}$-submodules of $\mfn = \mfg / \mfb^-$, where $\mfb^-$ is the
    negative Borel subalgebra of $\mfg$. Furthermore, $M$ and $N$ are isomorphic as
    $T^{\alpha}$-modules if and only if $-y^{-1} M$ and $-y^{-1} N$ are isomorphic
    as $T^{\alpha'}$-modules.  Finally, it is easy to see that if $M \subseteq \mfn$,
    and $N$ is an isomorphic $T^{\alpha'}$-module which is a
    $\mfg_{-\alpha'}$-submodule of $\mfn$, then $\Omega N = \tau(\Omega M,
    \alpha')$. Thus Carrell and Kuttler's formula in Equation
    (\ref{E:geomtranslate}) is equivalent to the formula in Equation
    (\ref{E:geom2comb}).
\end{proof}

It follows that we can apply the results of
Sections \ref{S:peterson} and \ref{S:ptgraph} in
this geometric setting.

\begin{example}\label{Ex:smooth}
    As mentioned in the introduction, in \cite{Sl13} it is shown that if $w$ is
    rationally smooth, then $\mcA(I(w))$ is free with coexponents equal to the     
    exponents of $w$. The exponents of $w$ are integers $m_1,\ldots,m_l$, where
    $l$ is the rank of $r$, such that the Poincare series $P_w(q) = \sum_{x
    \leq w} q^{\ell(x)}$ for the Bruhat interval below $w$ is equal to the
    product $\prod_i [m_i+1]_q$, where $[k]_q$ is the $q$-integer
    $(1+q+\ldots+q^{k-1})$.

    The elements listed in Table \ref{TBL:badpatlist} are all non-rationally
    smooth, and hence Corollary \ref{C:biconvexpat} gives another proof that
    $\mcA(I(w))$ is free when $w$ is rationally smooth. Corollary
    \ref{C:biconvexpat} does not imply that the coexponents of $\mcA(I(w))$ are
    equal to the exponents of $w$. However, $S(w) := -\Omega T_e X(w)$ is a
    lower order ideal, and when $X(w)$ is smooth a theorem of Akyildiz and
    Carrell states that the exponents of $w$ are equal to $\Exp S(w)$. In
    this situation we can use Corollary \ref{C:ptinversion} to show that the
    coexponents of $\mcA(I(w))$ are also equal to $\Exp S(w)$. Choose a sequence
    of element $e = y_k < y_{k-1} < \cdots < y_0 = w$, where $y_i = r_{\beta_i}
    y_{i-1}$ for some $\beta_i \in R^+$. As shown in \cite{CK03},
    $\tilde{\tau}(\Omega T_{y_i} X(w), \beta_i) = \Omega T_{y_{i+1}} X(w)$. 
    By Corollary \ref{C:ptinversion} and Equation (\ref{E:geom2comb}), if
    $\mcA(\Omega T_{y_i} X(w))$ is free then $\mcA( \Omega T_{y_{i+1}} X(w) )$
    is free with the same coexponents. Since $\mcA( \Omega T_e X(w))$ is free
    with coexpoents $\Exp S(w)$ by Theorem \ref{T:idealfree}, $\mcA(\Omega T_w
    X(w))$ has the same coexponents. 

    More generally, if $x \in W(R)$ is a smooth point of $X(w)$, then we can
    choose a sequence as in the last paragraph with $y_k = x$. It follows from
    Proposition \ref{P:petersonfree} that $-x^{-1} \Omega T_x X(w)$ is coconvex for
    every smooth point $x \in X(w)^T$, and if $\mcA(I(w))$ is free then
    $\mcA(\Omega T_x X(w))$ is free with the same coexponents. 
\end{example}

One difference between geometric and combinatorial Peterson translation is that
the weight spaces $\Omega M \subseteq \Omega T_y X$ can be partitioned into
positive and negative roots, and the positive roots are a subset of the inversion set $I(y)$. 
\begin{prop}\label{P:augmented}
    If $S$ is a coconvex subset containing an inversion set $I(y)$, and
    $\alpha \in I(y)$, then $I(r_{\alpha} y) \subseteq \tau(S, \alpha)$.
\end{prop}
\begin{proof}
    If $S \subseteq S'$, then $\tau(S,\alpha) \subseteq \tau(S', \alpha)$, so we
    can assume that $S = I(y)$ and $\alpha \in I(y)$. Suppose $\beta \in
    I(r_{\alpha} y)$, and let $U$ be the subspace of $V$ spanned by $\alpha$
    and $\beta$. Let $y_0$ be the unique element of $W(R_U)$ with $I(y_0) =
    I(y)_U$. Note that $r_{\alpha} \in W(R_U)$, and $\alpha \in I(y_0)$.
    Since $\beta \in I(r_{\alpha} y)$, either $r_{\alpha} \beta \in I(y)$, or
    $r_{\alpha} \beta \in R^-$ and $-r_{\alpha} \beta \not\in I(y)$. By
    construction  $r_{\alpha} \beta \in U$, so if $r_{\alpha} \beta \in I(y)$ then 
    $r_{\alpha} \beta \in I(y_0)$ as an element of $R_U$, and it follows that
    $\beta \in I(r_{\alpha} y_0) \subseteq R_U^+$. If $r_{\alpha} \beta \in R^-$
    and $-r_{\alpha} \beta \not\in I(y)$, then $r_{\alpha} \beta \in R^-_U$,
    and $-r_{\alpha} \beta \not\in I(y_0)$, so again $\beta \in I(r_{\alpha}
    y_0)$. 

    We can check that the proposition holds for inversion sets in rank $2$. Since $\beta \in I(r_{\alpha} y_0)$, 
    \begin{equation*}
        \beta \in \tau(I(y_0), \alpha) = \tau(I(y)_U, \alpha)
            \subseteq \tau(I(y), \alpha).
    \end{equation*}
\end{proof}

On the geometric side, Proposition \ref{P:augmented} states that if $S
\subseteq \Omega T_y X(w)$ contains $I(y)$ and $\alpha \in I(y)$ then
$\tilde{\tau}(S,\alpha)$ contains $I(r_{\alpha}
y)$. This suggests looking at pairs $(S, I(y))$,
where $S$ is a
coconvex subset of $R^+$ and $I(y)$ is an inversion set contained in $S$. The
pair $(S, I(y))$ corresponds to the subset $-y^{-1} S$ of $\Omega T_{y^{-1}} X$.
Peterson translation by $\alpha \in I(y)$ sends $(S, I(y))$ to $(\tau(S,
\alpha), I(r_\alpha y))$. 
\begin{example}
    In Example \ref{Ex:smooth}, we started with the pair $(S, I(w))$, where $S = I(w)$ and $X(w)$ was smooth. We then translated $(S, I(w))$ to the pair $(S(w), \emptyset)$ at the identity.
    In general we can translate any pair $(S, I(w))$ to some pair $(S_e,
    \emptyset)$ using only roots in the inversion set. However, $S_e$ is not
    necessarily a lower order ideal. For example, if we start with the element
    $w = s_1 s_2 s_3 s_2 s_1$ from Example \ref{Ex:typeAreflection} and set $S
    = I(w)$, then
    \begin{equation*}
        \tau((I(w), I(w)), \alpha_1 + \alpha_2 + \alpha_3) = (I(w), \emptyset)
    \end{equation*}
    sends $(I(w), I(w))$ to the identity in one step. It takes at least two
    steps to translate $I(w)$ to an ideal.
\end{example}

In addition to $I(w)$ and $S(w)$, there are many other ways to construct
subsets $S$ of $R^+$ from geometric constructions on the Schubert variety
$X(w)$. The next example gives a family of coconvex sets arising in this
way.
\begin{example}\label{Ex:maxsing}
    Let $\Theta_x X(w)$ denote the span of the reduced tangent cone to $X(w)$
    at $x$. If $x$ is a smooth point of $X(w)$, then $\Theta_x X(w) = T_x
    X(w)$.  Also, if $x \in W(R)$ is smooth or $G$ is simply-laced, then
    $\Theta_x X(w)$ is the space of tangents to $T$-invariant curves in $X(w)$,
    and 
    \begin{equation*}
        \Omega \Theta_x X(w) = \{\alpha \in R : x^{-1} \alpha \in R^- \text{ and } r_{\alpha} x\leq w\}
    \end{equation*}
    (see \cite{CK03} for more background). If $x \in W(R)$ is a maximal singularity
    then 
    \begin{equation*}
        \Theta_x X(w) = \sum_{\alpha \in R^+ \setminus I(x) } \tilde{\tau}(T_{r_{\alpha}x} X(w), \alpha)
    \end{equation*}
    by a theorem of Carrell and Kuttler \cite[Theorem 1.3]{CK06}. Since $x$ is
    a maximal singularity, the points $r_{\alpha} x$ in this sum are smooth,
    and consequently $- x^{-1} \tilde{\tau}(\Omega T_{r_{\alpha} x} X(w),
    \alpha)$ is coconvex for every $\alpha \in R^+ \setminus I(x)$ by Proposition
    \ref{P:petersonfree} and Example \ref{Ex:smooth}. Since the intersection of
    convex sets is convex, the union of coconvex sets is coconvex, and consequently
    \begin{equation*}
        -x^{-1} \Omega \Theta_x X(w) = \bigcup -x^{-1} \tilde{\tau}(\Omega
            T_{r_{\alpha} x} X(w), \alpha)
    \end{equation*}
    is coconvex at every maximal singularity $x$ of $X(w)$. Although the
    freeness of $\mcA(\Omega \Theta_x X(w))$ is theoretically resolved by
    Theorem \ref{T:coconvex}, it is an open question to fully characterize
    the pairs $(w,x)$ with this property.
\end{example}

The author is unaware of any geometric proof
of Proposition \ref{P:petersonfree}, and this
might also be an interesting problem. 

\bibliographystyle{amsalpha}
\bibliography{palindromic}

\end{document}